\documentclass[12pt]{amsart}
\usepackage{amsmath,amssymb,amsfonts,enumerate,amsthm}

\numberwithin{equation}{section}

\newtheorem{thm}{Theorem}[section]
\newtheorem{lem}[thm]{Lemma}
\newtheorem{cor}[thm]{Corollary}
\newtheorem{prop}[thm]{Proposition}
\newtheorem{rem}[thm]{Remark}
\newtheorem{defin}[thm]{Definition}

\newtheorem{remi}[thm]{Reminder}
\newtheorem{nota}[thm]{Notation}
\newtheorem{defen}[thm]{Definition}
\def\sm{ \smallskip}
\def\e{ \begin{enumerate} \it }
\def\ee{\end{enumerate} }
\def\en{\operatorname{end}}
\def\beg{\operatorname{beg}}
\def\height{\operatorname{height}}
\def\depth{\operatorname{depth}}
\def\reg{\operatorname{reg}}

\def\Ext{\operatorname{Ext}}
\def\Hom{\operatorname{Hom}}

\def\beg{\operatorname{beg}}
\def\en{\operatorname{end}}

\def\rank{\operatorname{rank}}
\def\gendeg{\operatorname{gendeg}}
\def\length{\operatorname{length}}

\def\dim{\operatorname{dim}}

\def\sk{\smallskip\par}
\def\mm{{\frak m}}

\def\fa{{\frak a}}

\newcommand{\m}{\mathfrak{m}}
\def\n{{\frak n}}

\newcommand{\Z}{\mathbb{Z}}
\newcommand{\fQ}{\mathbb{Q}}

\newcommand{\N}{\mathbb{N}}

\newcommand{\lO}{{\mathcal O}}
\newcommand{\cF}{{\mathcal F}}
\newcommand{\D}{{\mathcal D}}

\newcommand{\Q}{\mathbb{Q}}

\begin{document}
\bibliographystyle{amsplain}

\author{Markus Brodmann}
\address{University of Z\"urich, Institute of Mathematics, Winterthurerstrasse 190, 8057 Z\"urich. }
\email{brodmann@math.uzh.ch}

\author{Maryam Jahangiri}
\address{M.Jahangiri, Faculty of Basic Science, Tarbiat Modares University, Tehran, Iran.}
\email{m\_jahangiri@modares.ac.ir}

\author{Cao Huy Linh}
\address{Cao Huy Linh, Department of Mathematics, College of Education, Hue University, 32 Le Loi, Hue City, Vietnam.}
\email{huylinh2002@yahoo.com}

\thanks{\today }
\subjclass[2000]{}

\title[Castelnuovo-mumford regularity of deficiency modules ]
{Castelnuovo-mumford regularity of deficiency modules }

\begin{abstract}
Let $d \in \N$ and let $M$ be a finitely generated graded module
of dimension $\leq d$ over a Noetherian homogeneous ring $R$ with
local Artinian base ring $R_0$. Let $\beg(M)$, $\gendeg(M)$ and
$\reg(M)$ respectively denote the beginning, the generating
degree and the Castelnuovo-Mumford regularity of $M$. If $i \in
\N_0$ and $n \in Z$, let $d^i_M(n)$ denote the $R_0$-length of
the $n$-th graded component of the $i$-th $R_+$-transform module
$D^i_{R_+}(M)$ of $M$ and let $K^i(M)$ denote the $i$-th
deficiency module of $M$.

\sm

Our main result says, that $\reg(K^i(M))$ is bounded in terms of
$\beg(M)$ and the "diagonal values" $d^j_M(-j)$ with $j =
0,\cdots ,d-1$. As an application of this we get a number of
further bounding results for $\reg(K^i(M))$.

\end{abstract}
\maketitle
\section{Introduction} \sk
This paper is motivated by a basic question of projective
algebraic geometry, namely:

\smallskip

\e
\item[]\hskip1cm What bounds cohomology of a projective scheme?
\ee

\smallskip

ÔThe basic and initiating contributions to this theme are due to
Mumford [21] and Kleiman [18] (see also [13]). The numerical
invariant which plays a fundamental r$\hat{\rm{o}}$le in this
context, is the Castelnuovo-Mumford regularity, which was
introduced in [21]. Besides is foundational significance - in the
theory of Hilbert schemes for example- this invariant is the
basic measure of complexity in computational algebraic
geometry(s. [1]). This double meaning of (Castelnuovo-Mumford)
regularity made it  to one of the most studied invariants of
algebraic geometry. Notably a huge number of upper bounds for the
regularity have been established. We mention only a few more
recent references to such results, namely [1], [4], [7, 2], [10],
[11], [12], [19], [22].

\smallskip

It is also known, that Castelnuovo-Mumford regularity is closely
related to the boundedness - or finiteness- of cohomology at all.
More precisely, the regularity of deficiency modules provides
bounds for the so called cohomological postulation numbers, and
thus furnishes a tool to attack the finiteness problem for
(local) cohomology. This relation is investigated by Hoa-Hyry
[17] and Hoa [16] in the case of graded ideals in a polynomial
ring over a field. In [6] it was shown that for coherent sheaves
over projective schemes over a field $K$, cohomology is bounded
by the "cohomology diagonal". One challenge is to extend this
later result to the case where the base field $K$  is replaced by
an Artinian ring $R_0$ and hence to replace the bounds given in
[8] by "purely diagonal" ones. In the same spirit one could try
to generalize the results of Hoa and Hoa-Hyry. This is what we
shall do in the present paper.

\sm

Our basic result is a "diagonal bound" for the
Castelnuovo-Mumford regularity of deficiency modules.

\smallskip

To formulate this result we introduce a few notations. By
${\N}_0$ we denote the set of all non-negative integers, by $\N$
the set of all positive integers. Let $R:= \bigoplus_{n\geq
0}R_n$ be a Noetherian homogeneous ring with Artinian base ring
$R_0$ and irrelevant ideal $R_+:= \bigoplus_{n> 0}R_n$. Let $M$
be a finitely generated graded $R$-module. For each $i\in {\N}_0$
consider the graded $R$-module $D^i_{R_+}(M)$, where $D^i_{R_+}$
denotes the $i$-th $R_+$-transform functor, that is the $i$-th
right derived functor of the $R_+$-transform functor
$\displaystyle D_{R_+}(\bullet):=
\lim_{_{\stackrel{\longrightarrow}{n}}}\Hom_R((R_+)^n, \bullet
)$. In addition, for each $n\in \Z$ let $d^i_M(n)$ denote the
(finite) $R_0$-length of the $n$-th graded component of
$D^i_{R_+}(M)$. Moreover, let $\beg(M)$ and $\reg(M)$
respectively denote the beginning and the Castelnuovo-Mumford
regularity of $M$. If the (Artinian) base ring $R_0$ is local,
let $K^i(M)$ denote the $i$-th deficiency module of $M$. Fix
$d\in \N$ and $i\in \{0, \cdots, d\}$ and let $\dim(M)\leq d$.
Then, the announced bounding result says (s. Theorem 3.6):

\smallskip

\e
\item[]
\textit{The beginning $\beg(M)$ of $M$ and the cohomology diagonal
$(d^i_M(-i))_{i=0}^{d-1}$ of $M$ give an upper bound for the
regularity of $K^i(M)$.}
 \ee

\smallskip

This leads to a further bounding result for $\reg(K^i(M))$. To
formulate it, let $\reg^2(M)$ denote the Castelnuovo-Mumford
regularity of $M$ at and above level 2 and let $p_M$ denote the
Hilbert polynomial of $M$. Then (s. Theorem 4.2):

\smallskip

\e
\item[]  The invariant \it{$\reg(K^i(M))$ can be bounded in terms of
 the three invariants $\beg(M)$, $\reg^2(M)$ and
$p_M(\reg^2(M))$.} \ee

\smallskip

As a consequence we get (s. Corollary 4.4):

\smallskip

\e
\item[]\it{If
$\fa \subseteq R$ is a graded ideal, then $\reg(K^i(\fa))$
 and $\reg(K^i(R/\fa))$ can be bounded in terms of
$\reg^2(\fa)$, $\length(R_0)$, $\reg^1(R)$ and the number of
generating one-forms of $R$.} \ee

\smallskip

Applying this in the case where $R= K[x_1, \cdots, x_d]$ is a
polynomial ring over a field, we get an upper bound for
$\reg(K^i(R/\fa))$ which depends only on $d$ and $\reg^2(\fa)$.
This is a (slightly improved) version of a corresponding result
found in [17], which uses $\reg(\fa)$ instead of $\reg^2(\fa)$ as
a bounding invariant.

\smallskip

As an application of Theorem 4.2 (cf. (1.2)) we prove a few more
bounding results in the situation where $R=R_0[x_1, \cdots, x_d]$
is a polynomial ring over a local Artinian ring $R_0$, namely (s. Corrolaries 4.6, 4.8, 4.13):

\sm

\e
\item[]  \it{If $U\neq 0$ is a
finitely generated graded $R$-module and $M\subseteq U$ is a
graded submodule, then $\reg(K^i(M))$ and $\reg(K^i(U/ M))$ are
bounded in terms of $d$, $\length(R_0)$, $\beg(U)$, $\reg(U)$,
the number of generators of $U$ and the generating degree
$\gendeg(M)$}\it{of $M$.} \ee

\sm

\e
\item[] \it{If $F\stackrel{p}{\twoheadrightarrow}M$ is an epimorphism of graded
$R$-modules such that $F$ is
 free and of finite rank, then $\reg(K^i(M))$ is
bounded in terms of $d$, $ \length(R_0)$, $\beg(F)$,
$\gendeg(F)$, $\rank(F)$ and $\gendeg(\ker(p))$.} \ee

\sm

\e
\item[]  \it{Let $U$ and
$M$ be as above}\it{. Then $\reg(K^i(M))$ and $\reg(K^i(U/M))$
 are bounded in terms of $\length(R_0)$, $\beg(U)$,
$\reg(U)$, the Hilbert polynomial $p_{U}$ of $U$
 and the Hilbert polynomial $p_{U/M}$ of
$U/M$.} \ee

\sm
 For a fixed $i\in \N_0$ we consider the $i$-th
cohomological Hilbert function of the second kind
$d^i_M:\Z\rightarrow{\N}_0$ given by $n\mapsto d^i_M(n)$ and the
corresponding $i$-th cohomological Hilbert polynomial $q^i_M\in
\mathbb{Q}[x]$ so that $q^i_M(n)= d^i_M(n)$
for all $n\ll 0$. Based on these concepts we define the $i$-th
cohomological postulation number of $M$ by:
\[\nu^i_M:= \inf\{n\in \Z\mid \ q^i_M(n)\neq d^i_M(n)\}(\in \Z\cup \{\infty\}).\]
Now, let $d\in \N$ and let $\D^d$ be the class of all pairs $(R,
M)$ in which $R=\bigoplus_{n\in \N_0}R_n$ is a Noetherian
homogeneous ring with Artinian base ring $R_0$ and $M$ is a
finitely generated graded $R$-module of dimension $\leq d$. As a
first consequence of Theorem 3.6 we get, that for all pairs $(R,
M)\in \D^d$ and all $i\in \{0\cdots, d-1\}$ the cohomology
diagonal $( d^j_M(-j))_{j=0}^{d-1}$ of $M$ bounds the $i$-th
cohomological postulation number of $M$ (s. Theorem 5.3):

\sm

\e
\item[ ] There is a function $E^i_d: {\N}_0^d\rightarrow\Z$ such that for all
$x_0, \cdots, x_{d-1}\in {\N}_0$ and each pair $(R, M)\in \D^d$
such that $d^j_M(-j)\leq x_j$ for all $j\in \{0\cdots, d-1\}$  we
have
\[\nu^i_M\geq E^i_d(x_0, \cdots, x_{n}).\] \ee

\sm

This is indeed a generalization of the main result of [6] which
gives the same conclusion in the case where the base ring $R_0$ is
a field. Moreover, in our present proof, the bounding function
$E^i_d$ is defined much simpler than in [6].

\sm

As an application of Theorem 5.3 we show, that there are only
finitely many possible functions $d^i_M$ if the cohomology diagonal is fixed
(s. Theorem 5.4):

\sm

\e
\item[ ] Let $x_0, \cdots, x_{d-1}\in {\N}_0$. Then, the
set of functions
\[\{d^i_M\mid \ i\in \N_0, \ (R, M)\in \D^d: \ d^j_M(-j)\leq x_j \ \mbox{for} \ j=0, \cdots, d-1\}\]
 is finite. \ee

\sm

 \vskip 1 cm
\section{ Preliminaries}

In this section we recall a few basic facts which shall be
used later in our paper. We also prove a bounding result
for the Castelnuovo-Mumford regularity of certain graded modules.

\begin{nota}\label{2.1}
\rm{Throughout, let $R = \oplus_{n \geq 0}R_n$  be a homogeneous
Noetherian ring, so that $R$ is positively graded, $R_0$ is
Noetherian and $R = R_0[l_0,\cdots,l_r]$ with finitely many elements
$l_0,\cdots,l_r \in R_1$. Let $R_+$ denote the irrelevant ideal
$\oplus_{n > 0}R_n$ of $R$.}
\end{nota}

\begin{remi}\label{2.2}

\textit{(Local cohomology and Castelnuovo-Mumford regularity)} \rm{
\noindent(A)  Let $i \in {\N} _0:= \{0, 1, 2, \cdots \}.$ By
$H_{R_+}^i(\bullet)$ we denote the $i$-th local cohomology functor
with respect to $R_+$. Moreover by $D_{R_+}^i(\bullet)$ we denote
the $i$-th right derived functor of the ideal transform functor
$\displaystyle D_{R_+}(\bullet)= \lim _ {_{n \rightarrow \infty
}}((R_+)^n,\bullet)$ with respect to $R_+$.

\smallskip

\noindent(B) Let  $M:= \oplus_{n \in {\Z}}M_n$ be a  graded
$R$-module. Keep in mind that in this situation the $R$-modules
$H_{R_+}^i(M)$ and $D_{R_+}^i(M)$ carry natural gradings. Moreover
we then have a natural exact sequence of graded $R$-modules

\smallskip

\begin{enumerate}
\item[\ (i)]\hskip1cm $0\longrightarrow H_{R_+}^0(M)\longrightarrow
M\longrightarrow D_{R_+}^0(M)\longrightarrow
H_{R_+}^1(M)\longrightarrow 0$
\end{enumerate}

\smallskip

and natural isomorphisms of graded $R$-modules

\smallskip

\begin{enumerate}
\item[(ii)]\hskip1cm$D_{R_+}^i(M) \cong H_{R_+}^{i+1}(M) \ \ \
\text{for all}\ \ i
> 0.$
\end{enumerate}

\smallskip

\noindent(C) If $T$ is a graded $R$-module and $n\in {\Z }$, we use
$T_n$ to denote the $n$-th graded component of $T$. In particular,
we define the} {\it beginning} and the \it{end} \rm{of $T$
respectively by

\smallskip

\begin{enumerate}
\item[(i)]\hskip1cm$\beg(T):= \inf\{ n\in {\Z }| T_n \ne 0\}$,
\end{enumerate}

\smallskip

\begin{enumerate}
\item[(ii)]\hskip1cm$\en(T):= \sup\{ n\in {\Z }| T_n \ne 0\}$.
\end{enumerate}

\smallskip

with the standard convention that $\inf \emptyset = \infty$ and}
$\sup \emptyset = -\infty$.

\smallskip

\noindent(D) If the graded $R$-module $M$ is finitely generated, the
$R_0$-modules $H_{R_+}^{i}(M)_n$ are all finitely generated and
vanish as well for all $n \gg 0$ as for all $i \geq 0$. So, we have
\[-\infty \leq a_i(M):= \en(H_{R_+}^{i}(M)) < \infty \ \ \text{for all} \ \ i \geq 0\]
with $a_i(M):= -\infty$ for all $i\geq 0.$

\smallskip

If $k \in {\N} _0$, the \it{Castelnuovo-Mumford regularity of} $M$
\it{at and above level} $k$ \rm{is defined by}

\smallskip

\begin{enumerate}
\item[(i)] \hskip1cm$\reg^k(M): = \sup \{a_i(M) + i|\  i\geq k\}\,\,\, (< \infty )$,
\end{enumerate}

\smallskip

where as the \it {Castelnuovo-Mumford regularity} of $M$ \rm{is
defined by

\smallskip

\begin{enumerate}
\item[(ii)]\hskip1cm$\reg(M):= \reg^0(M).$
\end{enumerate}

\smallskip

(E) If $M$ is a graded $R$-module we denote the} {\it generating
degree of} $M$ by $\gendeg(M)$, thus

\smallskip

\begin{enumerate}
\item[(i)]\hskip1cm$\gendeg(M) = \inf\{n \in {\Z} | M = \bigoplus_{m \leq n}RM_m \}$.
\end{enumerate}

\smallskip

Keep in mind the well known relation (s. [9, 15.3.1])

\smallskip

\begin{enumerate}
\item[(ii)]\hskip1cm$\gendeg(M) \leq \reg(M).$
\end{enumerate}

\smallskip

\end{remi}

\begin{remi}\label{2.3} \it{(Cohomological Hilbert functions)}
\rm{  (A) Let $i\in {\N}_0$ and assume that the base ring $R_0$ is
Artinian. Let $M$ be a finitely generated graded $R$-module. Then,
the graded $R$-modules $H_{R_+}^{i}(M)$ are Artinian (s. [9,
7.1.4]). In particular for all $i\in {\N}_0$ and all $n\in {\Z}$
we may define the non-negative integers}

\smallskip

\begin{enumerate}
\item[(i)]\hskip1cm$h_M^i(n):= \length_{R_0}(H_{R_+}^{i}(M)_n)$,
\end{enumerate}

\smallskip

\begin{enumerate}
\item[(ii)]\hskip1cm$d_M^i(n):= \length_{R_0}(D_{R_+}^{i}(M)_n)$,
\end{enumerate}

\smallskip

\rm{Fix $i\in {\N}_0$. Then the functions}
\begin{enumerate}
\item[(iii)]\hskip1cm$h_M^i:{\Z}\rightarrow {\N}_0, \ \ n\mapsto h_M^i(n)$,
\end{enumerate}

\smallskip

\begin{enumerate}
\item[(iv)]\hskip1cm$d_M^i:{\Z}\rightarrow {\N}_0, \ \ n\mapsto d_M^i(n)$
\end{enumerate}

\rm{are called the $i$-th }\it {Cohomological Hilbert functions}
\rm{of the} \it{first} \rm{respectively the} \it{second kind}
\rm{of} $M$.

\smallskip

\noindent \rm{(B) Let $i\in {\N}_0$ and let $R$ and $M$ be as in
part (A). Then, there is a polynomial $p_M^i \in {\Q}[x]$ of
degree $<i$ such that (s.[9, 17.1.9])}

\smallskip

\begin{enumerate}
\item[(i)]\hskip1cm$p_M^i(n) = h_M^i(n) \ \ \text{for all}\ \ n\ll 0;$
\end{enumerate}

\smallskip

\begin{enumerate}
\item[(ii)]\hskip1cm$\deg(p_M^i)\leq i-1, \textrm{with equality if}  \ i= \dim(M).$
\end{enumerate}

\smallskip

\rm{We call $p_M^i$ the $i$-th} \it {Cohomological Hilbert
polynomial of the first kind of M}. \rm{Now, clearly by the
observation made in part (A) we also have polynomials $q_M^i \in
{\Q}[x]$ such that}

\smallskip

\begin{enumerate}
\item[(iii)]\hskip1cm$q_M^i(n) = d_M^i(n) \ \ \text{for all}\ \ n\ll 0.$
\end{enumerate}

\smallskip

\rm{These are called the} {\it Hilbert polynomials of the second
kind of $M$}. \rm{Observe that}

\smallskip

\begin{enumerate}
\item[(iv)]\hskip1cm$q_M^i = p_M^{i+1} \ \ \textrm{for all}\ \ i\in {\N}_0.$
\end{enumerate}

\smallskip

\rm{Finally, for all $i\in {\N}_0$ we define the $i$-th} \it
{cohomological postulation number of M} \rm{as

\smallskip

\begin{enumerate}
\item[(v)]\hskip1cm$\nu_M^i: = \inf\{n \in {\Z}| q_M^i(n) \ne d_M^i(n)\}(\in {\Z}\cup
\{\infty\}).$
\end{enumerate}

\rm{Observe that these numbers $\nu_M^i$ differ by 1 from the
cohomological postulation numbers introduced in [8].

\smallskip

\noindent(C) Let $R$ and $M$ be as in part (A). By $p_M \in
{\Q}[x]$ we denote the} {\it Hilbert polynomial} \rm{of} $M$.

\smallskip

\rm{By $p(M)$ we denote the}} {\it postulation number} $\sup\{n
\in {\Z} | \length_{R_0}(M_n)\ne p_M(n)\}$ \rm{of $M$.

\smallskip

Keep in mind that according to the Serre formula we have (s. [9,
17.1.6])}

\smallskip

\begin{enumerate}
\item[ ]\hskip1cm$\displaystyle p_M(n) = \sum_{i \geq 0}(-1)^id_M^i(n) = \length_{R_0}(M_n) - \sum_{j\geq
0}(-1)^jh_M^j(n).$
\end{enumerate}
\end{remi}

\begin{remi}\label{2.4}
\textit{(Filter regular linear forms)}  \rm{(A) Let $M$ be a
finitely generated graded $R$-module and let $x\in R_{1}$. By
NZD$_{R}(M)$ resp. ZD$_{R}(M)$ we denote the set of
non-zerodivisors resp. of zero divisors of $R$ with respect to
$M$.

\smallskip

The linear form $x\in R_{1}$ is said to be ($R_{+}$-) \it{filter
regular with respect to }$M$ \rm{if } $x\in
\mathrm{NZD}$$_{R}(M/\it{\Gamma}_{R_{+}}(M))$.

\smallskip

\noindent(B) Finally if $x\in R_{1}$ is filter regular with
respect to $M$ then the graded short exact sequences
\[0\longrightarrow (0:_Mx)\longrightarrow M\longrightarrow M/(0:_Mx)\longrightarrow0,\]
\[0\longrightarrow M/(0:_Mx)(-1)\longrightarrow M\longrightarrow M/xM\longrightarrow 0\]
imply}

\smallskip

\[\reg^{1}(M)\leq \reg(M/xM)\leq \reg(M).  \]
$\hfill \bullet$
\end{remi}

The following result will play a crucial role in the proof of our bounding result for the regularity of deficiency modules.

\begin{prop}\label{2.5}

Assume that the base ring $R_0$ is Artinian. Let $M$ be a finitely
generated graded $R$-module, let $x\in R_{1}$ be filter regular with
respect to $M$ and let $m\in \mathbb{Z}$ be such that $\reg(M/xM)
\leq m$ and $\gendeg((0:_{M}x)) \leq m$. Then
\[\reg(M) \leq m + h^{0}_{M}(m).\]

\end{prop}

\begin{proof}

By Reminder 2.4(B) we have $\text{reg}^{1}(M)\leq
\text{reg}(M/xM)\leq m$. So, it remains to show that
\[a_{0}(M)= \text{end}(H^{0}_{R_{+}}(M))\leq m+ h^{0}_{M}(m).\]
The short exact sequence of graded $R$-modules
\[0 \longrightarrow M/(0:_Mx)(-1) \overset{x}{  \longrightarrow} M \longrightarrow M/xM \longrightarrow 0\]
induces exact sequences of $R_{0}$-modules
\[0 \longrightarrow H^{0}_{R_+}(M/(0:_{M}x))_{n}\longrightarrow H^{0}_{R_+}(M)_{n+ 1}\longrightarrow H^{0}_{R_+}(M/xM)_{n+ 1}
\longrightarrow H^{1}_{R_+}(M/(0:_{M}x))_{n}\] for all $n\in
\mathbb{Z}$. As $H^{0}_{R_+}(M/xM)_{n+ 1}= 0$ for all $n\geq m$, we
thus get
\[H^{0}_{R_+}(M/(0:_{M}x))_{n}\cong H^{0}_{R_+}(M)_{n+ 1}\,\, \text{for all} \,\, n\geq m.\]
The short exact sequence of graded $R$-modules
\[0 \longrightarrow (0:_{M}x)\longrightarrow M\longrightarrow M/(0:_Mx) \longrightarrow 0\] together with the facts that
$H^{0}_{R_+}((0:_{M}x))= (0:_{M}x)$ and $H^{1}_{R_+}((0:_{M}x))=
0$ induce short exact sequences of $R_{0}$-modules
\[0\longrightarrow (0:_{M}x)_n\longrightarrow H^{0}_{R_+}(M)_n\longrightarrow H^{0}_{R_+}(M/(0:_{M}x))_{n}\longrightarrow 0\]
for all $n\in \mathbb{Z}$.

\smallskip

So, for all $n\geq m$ we get an exact sequence of $R_{0}$-modules
\[0\longrightarrow (0:_{M}x)_n\longrightarrow H^{0}_{R_+}(M)_n\stackrel{\pi_{n}}{\longrightarrow}
H^{0}_{R_+}(M)_{n+ 1}\longrightarrow 0.\] To prove our claim, we
may assume that $a_{0}(M)> m$. As end$((0:_{M}x))= a_{0}(M)$ and
$\gendeg((0:_{M}x))\leq m$ it follows that $(0:_{M}x)_{n}\ne 0$
for all integers $n$ with $m\leq n\leq a_{0}(M)$. Hence, for all
these $n$, the homomorphism $\pi_{n}$ is surjective but not
injective, so that $h^{0}_{M}(n)> h^{0}_{M}(n+ 1)$. Therefore,
for $n\geq m$ the function $n\mapsto h^{0}_{M}(n)$ is strictly
decreasing until it reaches the value $0$. Thus $h^{0}_{M}(n)= 0$
for all $n> m+ h^{0}_{M}(m)$, and this proves our claim.

\end{proof}


We now recall a few basic facts about deficiency modules and graded
local duality.


\begin{remi}\label{2.6}
\it{(Deficiency modules and local duality)}\rm{ (A) We assume that
the base ring $R_{0}$ is Artinian and local with maximal ideal
$\mm_{0}$. As $R_{0}$ is complete it is a homomorphic image of a
complete regular ring $A_{0}$. Factoring out an appropriate system
of parameters of $A_{0}$ we thus may write $R_{0}$ as a homomorphic
image of a local Artinian Gorenestein ring $(S_{0}, \n_{0})$. Let
$d'$ be the minimal number of generators of the $R_{0}$-module
$R_{1}$ and consider the polynomial ring $S:= S_{0}[x_{1}, ...,
x_{d'}]$. Then, we have a surjective homomorphism
$S\stackrel{f}{\twoheadrightarrow} R$ of graded rings.

\smallskip

For all $i\in {\N} _0$ and all finitely generated graded
$R$-modules, the $i$-th } \it{deficiency module of} $M$ \rm{is
defined as the finitely generated graded $R$-module (cf [23,
Section 3.1] for the corresponding concept for a local Noetherian
ring $R$ which is a homomorphic image of a local Gornestein ring
$S$.)

\smallskip

\begin{enumerate}
\item[(i)] \hskip1cm$K^{i}(M):= \Ext^{d'- i}_{S}(M, S(-d')).$
\end{enumerate}

\smallskip

The module

\smallskip

\begin{enumerate}
\item[(ii)]\hskip1cm$K(M):= K^{\text{dim(M)}}(M)$
\end{enumerate}

\smallskip

is called the } \it{canonical module of $M$.}

\smallskip

\noindent\rm{(B) Keep the previous notations and hypotheses. Let
$E_{0}$ denote the injective envelope of the $R_{0}$-module
$R_{0}/\m_{0}$. Then, by  Graded Matlis Duality and the  Graded
Local Duality Theorem (s. [9, 13.4.5] for example) we have}

\smallskip

\begin{enumerate}
\item[] $\length_{R_{0}}(K^{i}(M)_{n})= h^{i}_{M}(-n)$
\end{enumerate}

\smallskip

\rm{for all $i\in {\N} _0$ and all $n\in \mathbb{Z}$.}

\smallskip

\noindent\rm{(C) As an easy consequence of the last observation we
now get the following relations for all $i\in {\N} _0$ and all
$n\in \mathbb{Z}$:}

\smallskip

\begin{enumerate}
\item[(i)]\hskip1cm  $d^{i}_{M}(n)= \length_{R_{0}}(K^{i+ 1}(M)_{-n}),\ \ \mbox{if} \ \ i>0 \ \ \mbox
{and}$
\item[]\hskip1cm $d^{0}_{M}(n)\geq \length_{R_{0}}(K^{1}(M)_{-n})$ with equality if $n< \beg(M)$;
\end{enumerate}

\smallskip

\begin{enumerate}
\item[(ii)]\hskip1cm$p^{i}_{M}(n)= p_{K^{i}(M)}(-n)$;
\end{enumerate}

\smallskip

\begin{enumerate}
\item[(iii)]\hskip1cm$q^{i}_{M}(n)= p_{K^{i+ 1}(M)}(-n)$;
\end{enumerate}

\smallskip

\begin{enumerate}
\item[(iv)]\hskip1cm$a_{i}(M)= -\beg(K^{i}(M))$;
\end{enumerate}

\smallskip

\begin{enumerate}
\item[(v)]\hskip1cm$\en(K^i(M))= -\beg(H^i_{R_+}(M))$;
\end{enumerate}

\smallskip

\begin{enumerate}
\item[(vi)]\hskip1cm$\nu^{i}_{M}= -p(K^{i+ 1}(M)).\hfill \bullet $

\end{enumerate}

\smallskip

\end{remi}

\vskip 1 cm
\section{ Regularity of Deficiency Modules}

We keep the notations introduced in Section 2. Throughout this
section we assume in addition that the Noetherian homogenous ring $
R=\bigoplus _{n \geq0} R_{n}$ has Artinian local base ring $(R_{0},
\mm_0)$.

\smallskip

The aim of the present section is to show that the
Castelnuovo-Mumford regularity of the deficiency modules $K^i(M)$ of
the finitely generated graded $R$-module $M$ is bounded in terms of
the beginning $\beg(M)$ of $M$ and the "cohomology diagonal"
$(d^{i}_{M}(-i))_{i=0}^{\dim(M)-1}$ of $M$.

We first prove three auxiliary results.

\begin{lem}\label{3.1}
$\depth (K^{\dim(M)}(M))\geq \min\{2, \dim(M)\}.$
\end{lem}

\begin{proof}
In the notation introduced in Reminder (2.6) we have
$K(M)_{\m}\cong K(M_{\m})$. As $M$ and $K(M)$ are finitely
generated graded $R$-modules we have $\dim(M)= \dim(M_{\m})$ and
$\depth (K(M))= \depth (K(M)_{\m})$. Now, we conclude by [23,
Lemma 3.1.1(C)].
\end{proof}

\begin{lem}\label{3.2}

Let $x\in R_1$ be filter regular with respect to $M$ and the modules $K^j(M)$. Then, there are
short exact sequences of graded $R$-modules
\[0\longrightarrow (K^{i+1}(M)/xK^{i+1}(M))(+1)\longrightarrow K^{i}(M/xM)\longrightarrow
(0:_{K^{i}(M)}x)\longrightarrow 0\]
\end{lem}

\begin{proof}
In the local case, this result is shown in [24, Proposition 2.4].
In our graded situation, one may conclude in the same way.
\end{proof}

\begin{lem}\label{3.3}
Let $i\in {\N}_0$ and $n \geq i$. Then
\[\length_{R_0}(K^{i+1}(M)_n) \leq\sum_{j=0}^{i}\binom{n-j-1}{i-j}\biggr[\sum_{l=0}^{i- j}
\binom{i-j}{l}d^{i- l}_{M}(l- i)\biggr].\]
\end{lem}

\begin{proof}
According to [8, Lemma 4.4] we have a corresponding inequality
with $d^{i}_{M}(-n)$ on the lefthand side. Now, we conclude by
Reminder 2.6(C)(i).
\end{proof}
Next we recursively define a class of bounding functions.
\begin{defin}\label{3.4}
\rm{For $d\in {\N}_0$ and $i\in \{0,\cdots, d\}$ we define the functions
\[F^i_d: {\N}_0^d\times {\Z}\longrightarrow  {\Z}\]
as follows: In the case $i= 0$ we simply set

\smallskip

\begin{enumerate}
\item[(i)]\hskip1cm $F^0_d(x_0,\cdots, x_{d-1}, y):= -y$.
\end{enumerate}
\smallskip
Concerning the case $i= 1$ we set
\begin{enumerate}

\smallskip

\item[(ii)]\hskip1cm $F^1_1(x_0, y):= 1- y$ and
\item[(iii)] \hskip1cm$F^1_d(x_0,\cdots, x_{d-1}, y):= \max\{0, 1- y\}+ \sum_{i=0}^{d- 2}
\binom{d-1}{i}\ x_{d- i- 2}$, for $d\geq 2$.
\end{enumerate}
\smallskip
In the case $i= d= 2$ we define

\smallskip

\begin{enumerate}
\item[(iv)]\hskip1cm $F^2_2(x_0, x_1, y):= F^1_2(x_0, x_1, y)+ 2$.
\end{enumerate}

\smallskip

If $d\geq 3$ and $2\leq i\leq d- 1$ and under the assumption that
$F^{i- 1}_{d- 1}$, $F^{i}_{d- 1}$ and $F^{i- 1}_{d}$ are already
defined, we first set

\smallskip

\begin{enumerate}
\item[(v)]\hskip1cm $ m_i:= \max\{F_{d-1}^{i-1}(x_0+x_1,\cdots,x_{d-2}+ x_{d-1},y), F_d^{i-1}(x_0,\cdots,x_{d-1},y)+1\} +1$,

\smallskip

\item[(vi)] \hskip1cm$n_i:= F_{d-1}^{i}(x_0+x_1,\cdots,x_{d-2}+ x_{d-1},y)$,

\smallskip

\item[(vii)]\hskip1cm $t_i:= \max\{m_i, n_i\}$.
\end{enumerate}

\smallskip

Then, using this notation we define

\smallskip

\begin{enumerate}
\item[(viii)]\hskip1cm $F_{d}^{i}(x_0,\cdots,x_{d-1},y):= t_i + \sum_{j = 0}^{i-1}\binom{t_i-j-1}{i-j-1}\Delta_{ij}$,
\end{enumerate}

\smallskip

where $\Delta_{ij} = \sum_{l =0}^{i-j-1}\binom{i-j-1}{l}x_{i-l-1}$.

\smallskip

Finally, assuming that $d\geq 3$ and that $F_{d-1}^{d-1}$ and
$F_{d}^{d-1}$ are already defined, we set}

\smallskip

\begin{enumerate}
\item[(ix)]\hskip1cm $F_{d}^{d}(x_0,\cdots, x_{d-1},y):=$ \\

\smallskip

$\ \ \ \ \ \ \ \ \ \ \ \max\{F_{d-1}^{d-1}(x_0+x_1,\cdots,x_{d-2}+
x_{d-1},y), F_d^{d-1}(x_0,\cdots,x_{d-1},y)+1\} +1$.\hfill $\bullet
$
\end{enumerate}
\end{defin}

\begin{rem}\label{3.5}
\rm{(A) Let $d\in {\N}_0$ and $i\in \{0,\cdots, d\}$. Let
$(x_0,\cdots,x_{d-1}, y), (x_0',\cdots,x_{d-1}', y')\in
{\N}_0^d\times {\Z}$ such that
\[x_i\leq x_i'\,\,\,\textrm{for all}\,\,\,i\in \{0,\cdots, d-1\}\,\,\,\textrm{and}\,\,\,y'\leq y.\]
Then it follows easily by induction on $i$ and $d$ that
\[F^i_d(x_0,\cdots, x_{d-1}, y)\leq F^i_d(x_0',\cdots, x_{d-1}', y').\]

\smallskip

\noindent(B) It also follows by induction on $i$, that the auxiliary
numbers $m_i$ and $t_i$ of Definition 3.4 all satisfy the inequality
$\min\{m_i, t_i\} \geq i$.

\smallskip

\noindent(C) Let $s, d\in {\N}$ with $s\leq d$ and let $i\in
{\N}_0$ with $i\leq s$. Moreover, let $(x_0,\cdots,x_{s-1}, y)\in
{\N}^s\times {\Z}$. We then easily obtain by induction on $i$,
that
\[F^{i}_{s}(x_0,\cdots, x_{s-1},y)\leq F_{d}^{i}(x_0,\cdots,
x_{s-1},0,\cdots,0 , y).\]}\hfill $\bullet $
\end{rem}
Now we are ready to state the main result of the present section.

\begin{thm}\label{3.6}
Let $d\in {\N}$, $i\in \{0,\cdots, d\}$ and let $M$ be a finitely generated graded $R$-module such that $\dim(M)= d$. Then
\[\reg(K^i(M))\leq F^{i}_{d}(d^0_M(0), d^1_M(-1),\cdots,d^{d-1}_M(1-d), \beg(M)).\]
\end{thm}

\begin{proof}
We shall proceed by induction on $i$ and $d$. As $\dim(K^0(M))\leq
0$ and in view of Reminder 2.6(C)(v) we first have
\begin{align*}
\reg (K^0(M))= \en(K^0(M))= - \beg(H_{R_+}^0(M))&\leq -\beg(M) \\
&= F_d^0(d_M^0(0),\cdots,d_M^{d-1}(1-d), \beg(M)).
\end{align*}

This proves the case where $i= 0$.

\smallskip

So, let $i> 0$. We may assume that $R_0/\mm_0$ is infinite. In
addition, we may replace $M$ by $M/H^0_{R_+}(M)$ and hence assume that
$\depth(M)> 0$.

\smallskip

Let $x\in R_1$ be a filter regular element with respect to $M$ and
all the modules $K^j(M)$. Observe that $x\in$ NZD$(M)$. By Lemma
3.2 we have the exact sequences of graded $R$-modules
\[\hskip1.4cm 0\longrightarrow (K^{j+1}(M)/xK^{j+1}(M))(+1)\longrightarrow
K^j(M/xM)\longrightarrow (0:_{K^j(M)}x)\longrightarrow
0\hskip1.4cm(1)\]

for all $j\in {\N}_0$.

\smallskip

Since $\depth(M)> 0$ we have $K^0(M)= 0$. So, the sequence (1)
yields an isomorphism of graded $R$-modules
\[\hskip4.5cm (K^{1}(M)/xK^{1}(M))(+1)\cong
K^0(M/xM).\hskip4.5cm (2)\]
As $\dim (K^0(M/xM))\leq 0$ the isomorphism (2) and Reminder
2.6(C)(v) imply
\[\reg(K^{1}(M)/xK^{1}(M))= \reg(K^0(M/xM))+ 1= \en(K^0(M/xM))+1 \]
\[=1- \beg(H^0_{R_+}(M/xM))\leq 1- \beg(M/xM)
\leq 1- \beg(M).\]
Therefore,
\[\hskip4.8cm \reg(K^{1}(M)/xK^{1}(M))\leq 1-\beg(M).\hskip4.7cm (3)\]
Assume first that $d= \dim(M)= 1$. Then, by Lemma 3.1 we have
$\depth(K^{1}(M))\geq \min\{2, \dim(M)\}= 1$, whence
$\reg(K^{1}(M))= \reg^1(K^{1}(M))$. It follows that (cf. Reminder
2.4(B))
\[\reg(K^{1}(M))\leq \reg(K^{1}(M)/xK^{1}(M))\leq 1- \beg(M)= F^1_1(d^0_M(0), \beg(M)).\]
This proves our result if $d= 1$.

\smallskip

So, from now on we assume that $d\geq 2$. We first focus to the case
$i= 1$ and consider the exact sequence (1) for $j= 1$, hence
\[\hskip1.8cm 0\longrightarrow (K^{2}(M)/xK^{2}(M))(+1)\longrightarrow
K^{1}(M/xM)\longrightarrow (0:_{K^{1}(M)}x)\longrightarrow
0.\hskip1.7cm (4)\]
If $d= \dim(M)= 2$, we have $\dim(M/xM)= 1$, and so by the case
$d=1$ we get
\[\reg(K^{1}(M/xM))\leq 1- \beg(M/xM)\leq 1- \beg(M).\]
From (4) and Reminder 2.2(E)(ii) it follows that
\[\gendeg((0:_{K^{1}(M)}x))\leq \reg(K^{1}(M/xM))\leq 1- \beg(M).\]
Set $m_0:= 1- \beg(M)$. If $m_0\leq 0$, by the inequality (3),
Proposition 2.5 (applied with $m= 0$) and Reminder 2.6(C)(i) we
obtain
\begin{align*}
\reg(K^{1}(M))&\leq 0+ h^0_{K^{1}(M)}(0)\\
&\leq \length(K^{1}(M)_0)\\
&=d^0_M(0).
\end{align*}
If $m_0> 0$ we have $d^0_M(-m_0)\leq d^0_M(0)$. So, by (3),
Proposition 2.5 and Reminder 2.6(C)(i) we get
\begin{align*}
\reg(K^{1}(M))&\leq m_0+ h^0_{K^{1}(M)}(m_0)\\
&\leq m_0+\length(K^{1}(M)_{m_0})\\
&= 1- \beg(M)+ d^0_M(-m_0)\\
&\leq 1-\beg(M)+d^0_M(0).
\end{align*}
So, (cf. Definition 3.4(iii))
\begin{align*}
\reg(K^{1}(M))&\leq \max\{d^0_M(0), 1- \beg(M)+ d^0_M(0)\}\\
&\leq \max\{0, 1- \beg(M)\}+ d^0_M(0)\\
&= F^1_2(d^0_M(0), d^1_M(-1), \beg(M)).
\end{align*}
This proves the case $d= 2, i= 1$.

\smallskip

If $d\geq 3$, by induction on $d$, we have (cf. Definition
3.4(iii))
\begin{align*}
\reg (K^1(M/xM))&\leq
F_{d-1}^1(d_{M/xM}^0(0),\cdots,d_{M/xM}^{d-2}(2-d), \beg(M/xM))\\
&= \max\{0, 1-\beg(M/xM)\} +
\sum_{i=0}^{d-3}\binom{d-2}{i}d_{M/xM}^{d-i-3}(i+3-d) \\
&\leq \max\{0, 1-\beg(M)\} +
\sum_{i=0}^{d-3}\binom{d-2}{i}[d_{M}^{d-i-3}(i+3-d)+
d_{M}^{d-i-2}(i+2-d)].
\end{align*}
Set
\[t_0 := \max\{0, 1- \beg(M)\} + \sum_{i=0}^{d-3}\binom{d-2}{i}[d_{M}^{d-i-3}(i+3-d)+ d_{M}^{d-i-2}(i+2-d)].\]
By the exact sequence (4) and Reminder 2.2(E)(ii) we now get
\[ \gendeg((0:_{K^1(M)}x))\leq \reg (K^1(M/xM))\leq t_0.\]
By (3) we also have $\reg(K^{1}(M)/xK^{1}(M))\leq t_0$. As
$t_0\geq 0$, we have $d_{M}^{0}(-t_0)\leq d_{M}^{0}(0)$. So, by
Proposition 2.5 and Reminder 2.6(C)(i) we obtain
$$\reg (K^1(M))\leq t_0 + h^0_{K^1(M)}(t_0)\leq t_0 +
\length(K^1(M)_{t_0}) \leq t_0+ d_{M}^{0}(-t_0) \leq t_0+
d_{M}^{0}(0)$$
\begin{align*}
&= \max\{0, 1- \beg(M)\} +
\sum_{i=0}^{d-3}\binom{d-2}{i}[d_{M}^{d-i-3}(i+3-d)+ d_{M}^{d-i-2}(i+2-d)] + d_M^0(0)\\
& \leq \max\{0, 1- \beg(M)\} +
\sum_{i=0}^{d-2}\binom{d-1}{i}d_M^{d-i-2}(i+2-d).
\end{align*}
From this we conclude that (cf. Definition 3.4(iii))
\[\reg (K^1(M)) \leq F_{d}^{1}(d_M^0(0),\cdots,d_M^{d-1}(1-d), \beg(M)).\]
So, we have done the case $i=1$ for all $d\in {\N}$.

We thus attack now the case with $i\geq 2$. First, let $d=2$.
Then, in view of the sequence (4), by the fact that $x$ is filter
regular with respect to $K^1(M)$ and by what we have already
shown in the cases $d\in \{1, 2\}$ and $i= 1$, we get
\begin{align*}
\reg (K^{2}(M)/xK^{2}(M)) &\leq \max\{\reg (K^1(M/xM)), \reg((0:_{K^1(M)}x)) + 1\} + 1\\
&\leq \max\{\reg (K^1(M/xM)), \reg (K^1(M)) + 1\} + 1\\
&\leq \max\{1 - \beg(M), \max\{0, 1 - \beg(M)\}+ d_M^0(0)+ 1\} + 1\\
&\leq \max\{0, 1 - \beg(M)\} + d_M^0(0) + 2.
\end{align*}
As $\depth (K^2(M)) \ge \min\{2, \dim(M)\}$ (s. Lemma 3.1) we have
$\depth(K^2(M)) = 2$, thus $\reg(K^{2}(M))= \reg^1(K^{2}(M))$.
Hence (cf. Reminder 2.4(B) and Definition 3.4(iii),(iv))
\begin{align*}
\reg (K^2(M)) &\leq \reg (K^{2}(M)/xK^{2}(M)) \\
&\leq  \max\{0, 1 - \beg(M)\} + d_M^0(0) + 2\\
&= F_2^2(d_M^0(0), d_M^1(-1), \beg(M)).
\end{align*}
This completes the case $d= 2$. So, let $d>2$.

\smallskip

By induction on $d$ and in view of Remark 3.5(A) we have
\begin{align*}
\reg (K^k(M/xM)) &\leq F_{d-1}^k(d_{M/xM}^0(0),d_{M/xM}^{1}(-1),\cdots, d_{M/xM}^{d-2}(2-d), \beg(M/xM))\\
&\leq F_{d-1}^k(d_{M}^0(0)+ d_{M}^1(-1),\cdots, d_{M}^{d-2}(2-d)+
d_{M}^{d-1}(1-d),\beg(M)),
\end{align*}
for $0\leq k\leq d-1$.

\smallskip

Therefore
\[\reg (K^k(M/xM))\leq F^k_{d-1}(d_{M}^0(0)+
d_{M}^1(-1),\cdots,d_{M}^{d-2}(2-d)+d_{M}^{d-1}(1-d), \beg(M))\]
\[\hskip5.8cm\textrm{for all}\ \ k\in \{0,\cdots,d-1\}.\hskip5.7cm (5)\]
We first assume that $2\leq i\leq d-1$. Then, by induction on $i$ we
have
\[\hskip2.3cm\reg (K^{i-1}(M)) \leq F_d^{i-1}(d_M^0(0),d_M^{1}(-1),\cdots,
d_M^{d-1}(1-d), \beg(M)).\hskip2.3cm(6)\]
If we apply the exact sequence (1) with $j= i-1$ and keep in mind
that $x$ is filter regular with respect to $K^{i-1}(M)$ we thus
get by (5) and (6):
\begin{align*}
\reg (K^{i}(M)/xK^{i}(M)) &\leq \max\{\reg (K^{i-1}(M/xM)), \reg((0:_{K^{i-1}(M)}x)) + 1\} + 1\\
& \leq \max\{\reg (K^{i-1}(M/xM)), \reg (K^{i-1}(M))+ 1\} + 1\\
&\leq\max\{F_{d-1}^{i-1}(d_{M}^0(0)+ d_{M}^{1}(-1),\cdots,
d_{M}^{d-2}(2-d)+d_{M}^{d-1}(1-d),
\beg(M)),\\
&\ \ \ \ \ \ \ \ \ \ \ F_d^{i-1}(d_M^0(0),d_M^{1}(-1), \cdots,
d_M^{d-1}(1-d),\beg(M)) + 1\} + 1.
\end{align*}
If we apply the sequence (1) with $j= i$, we
obtain
\[\gendeg((0:_{K^{i}(M)}x)) \leq \reg (K^{i}(M/xM)).\]
According to (5) we have the inequality
\[\reg (K^{i}(M/xM))\leq F_{d-1}^{i}(d_M^0(0)+ d_M^1(-1),\cdots,
d_M^{d-2}(2-d)+ d_M^{d-1}(1-d), \beg(M)). \]
Set
\begin{align*}
m_i:&= \max\{F_{d-1}^{i-1}(d_{M}^0(0)+ d_{M}^{1}(-1),\cdots, d_{M}^{d-2}(2-d)+ d_{M}^{d-1}(1-d), \beg(M)), \\
&\ \ \ \ \ \ \ \ \ \ \ \ \ \ \ \ F_d^{i-1}(d_M^0(0),d_M^{1}(-1),\cdots, d_M^{d-1}(1-d), \beg(M)) + 1\} + 1,\\
n_i:& =F_{d-1}^{i}(d_M^0(0)+ d_M^1(-1),\cdots, d_M^{d-2}(2-d)+ d_M^{d-1}(1-d), \beg(M)), \text{ and}\\
t_i:&= \max\{m_i, n_i\}.
\end{align*}
Note that by Remark 3.5 (B) we have $t_i\geq i$. Hence, by
Proposition 2.5 and Lemma 3.3
\begin{align*}
\reg (K^{i}(M)) &\leq t_i + h^0_{K^i(M)}(t_i)\\
&\leq t_i + \length(K^{i}(M)_{t_i})\\
&\leq t_i + \sum_{j =0}^{i-1}\binom{t_i-j-1}{i-j-1}\biggr[ \sum_{l =
0}^{i-j-1}\binom{i-j-1}{l}d_M^{i-l-1}(l-i+1)\biggr].
\end{align*}
Thus, we obtain (cf. Definition 3.4(viii))
\[\reg (K^{i}(M)) \leq F_d^{i}(d_M^0(0),d_M^{1}(-1),\cdots, d_M^{d-1}(1-d), \beg(M)).\]
This completes the case where $i\leq d-1$. It thus remains to treat
the cases with $i=d>2$.

\smallskip

Now, by Lemma 3.1 we have $\depth (K^d(M)) \geq 2$. So, again by
Reminder 2.4(B) and by use of the sequence (1) we get
\begin{align*}
\reg (K^d(M)) &\leq \reg (K^{d}(M)/xK^{d}(M))\\
&\leq \max\{\reg (K^{d-1}(M/xM)), \reg((0:_{K^{d-1}(M)}x)) + 1\} + 1\\
&\leq \max\{\reg (K^{d-1}(M/xM)), \reg (K^{d-1}(M)) + 1\} +1.
\end{align*}
By induction and Remark 3.5(A) it holds
\[\reg (K^{d-1}(M/xM))\leq
F_{d-1}^{d-1}(d_{M/xM}^0(0),d_{M/xM}^{1}(-1),\cdots,
d_{M/xM}^{d-2}(2-d), \beg(M/xM)) \]
\[\leq F_{d-1}^{d-1}(d_M^0(0)+
d_M^1(-1), d_M^{1}(-1)+ d_M^2(-2),\cdots, d_M^{d-2}(2-d)+
d_M^{d-1}(1-d), \beg(M)).\]
By the case $i= d-1$ we have
\[\reg (K^{d-1}(M))\leq F_d^{d-1}(d_M^0(0),d_M^{1}(-1),\cdots, d_M^{d-1}(1-d), \beg(M)).\]
This implies that (cf. Definition 3.4(ix))
\begin{align*}
\reg (K^d(M)) &\leq \max\{F_{d-1}^{d-1}(d_M^0(0)+d_M^1(-1),\cdots,d_M^{d-2}(2-d)+ d_M^{d-1}(1-d), \beg(M)),\\
&\ \ \ \ \ \ \ \ \ \   F_d^{d-1}(d_M^0(0),\cdots,d_M^{d-1}(1-d), \beg(M))+1\} +1\\
&= F_d^{d}(d_M^0(0),\cdots,d_M^{d-1}(1-d), \beg(M)).
\end{align*}
So, finally we may conclude that
\[\reg (K^i(M))\leq F_d^{i}(d_M^0(0),d_M^{1}(-1),\cdots, d_M^{d-1}(1-d), \beg(M)),\]
for all $d\in {\N}$ and all $i\in \{0,\cdots,d\}$.
\end{proof}

\begin{cor}\label{3.7}
Let $d\in {\N}$, $i\in \{0,\cdots,d\}$, $(x_0,\cdots,x_{d-1}, y)\in
{\N}^d_0\times {\Z}$ and let $M$ be a finitely generated graded
$R$-module such that $\dim(M)\leq d$, $d^j_M(-j)\leq x_j$ for all
$j\in \{0,\cdots,d-1\}$ and $\beg (M)\geq y$. Then
\[\reg (K^i(M))\leq F_d^{i}(x_0,\cdots,x_{d-1}, y).\]
\end{cor}
\begin{proof}
If $M=0$, we have $K^i(M)=0$ and so our claim is obvious.

\smallskip

If $\dim(M)=0$, we have $K^i(M)=0$ for all $i>0$ and
$\dim(K^0(M))\leq0$ so that (s. Reminder 2.6(C)(v))
\[\reg(K^0(M))= \en(K^0(M))=- \beg(H^0_{R_+}(M))=- \beg(M)\leq -y=F^0(x_0,\cdots,x_{d-1}, y).\]
So, it remains to show our claim if $\dim(M)>0$. But now, we may
conclude by Theorem 3.6 and Remark 3.5(A), (C).
\end{proof}

\vskip 1 cm
\section{ Bounding $\reg(K^i(M))$ in terms of $\reg^2(M)$}

We keep the notations introduced in section 3. In particular we
always assume that the homogeneous Noetherian ring $R=\bigoplus
_{n\geq 0}R_n$ has Artinian local base ring $(R_0, \mm_0)$. We
have seen in the previous section, that the Castelnuovo-Mumford
regularity of the deficiency modules $K^i(M)$ of a finitely
generated graded $R$-module $M$ is bounded in terms of the
invariants $d^j_M(-j)$ $(j=0,\cdots,\dim(M)-1)$ and $\beg(M)$. We
shall use this result in order to bound the numbers $\reg(K^i(M))$
in terms of the Castelnuovo-Mumford regularity of $M$. This idea
is inspired by Hoa-Hyry [17] who gave similar results for graded
ideals in a polynomial ring over a field.

\smallskip

As an application we shall derive a number of further bounds on the
invariants $\reg(K^i(M))$.

\begin{defen}\label{4.1}
\rm{Let $d\in \N$ and $i\in \{0,\cdots,d\}$. We define a bounding
function
\[G^i_d:{\N}_0\times {\Z}^2\rightarrow \Z\]
by
\[G^i_d(u, v, w):= F^i_d(u, 0,\cdots, 0, v-w)- w.\]
Now, we are ready to give a first result of the announced type. It
says that the numbers $\reg(K^i(M))$ find upper bounds in terms of
$\reg^2(M)$ and the Hilbert polynomial of $M$.\hfill $\bullet $}
\end{defen}

\begin{thm}\label{4.2}
Let $p\in {\N}_0$, $d\in \N$, $i\in \{0,\cdots,d\}$, $b, r\in \Z$
and let $M$ be a finitely generated graded $R$-module with
$\dim(M)\leq d$, $\beg(M)\geq b$, $\reg^2(M)\leq r$ and $p_M(r)\leq
p$. Then
\[\reg(K^i(M))\leq G^i_d(p, b, r).\]
\end{thm}

\begin{proof}
Observe that $\beg(M(r))\geq b-r$. On use of Corollary 3.7 we now
get
\begin{align*}
\reg(K^i(M))+ r= \reg(K^i(M)(-r))&= \reg(K^i(M(r)))\\
&\leq F^i_d(d^0_{M(r)}(0),d^0_{M(r)}(-1), \cdots,d^{d-1}_{M(r)}(1-d), b-r)\\
&=F^i_d(d^0_M(r), d^1_M(r-1), \cdots,d^{d-1}_M(r+1-d), b-r).
\end{align*}
For all $j\in \N$ we have $d^j_M(r-j)= h^{j+1}_M(r-j)=0$, so that
\[d^1_M(r-1)=\cdots=d^{d-1}_M(r+1-d)= 0.\]
In addition $d^j_M(r)=h^{j+1}_M(r)=0$ for all $j\in \N$, which
implies that $d^0_M(r)=p_M(r)\leq p$ (s. Reminder 2.3(C)). In
view of Remark 3.5(A) the above inequality now induces
\[ \reg(K^i(M))+r\leq F^i_d(p,0,\cdots,0,b-r)\]
and this proves our claim.
\end{proof}

Bearing in mind possible application to Hilbert schemes for example
one could ask for bounds which apply uniformly to all graded
submodules $M$ of a given finitely generated graded $R$-module $U$
and depend only on basic invariants of $M$.

\smallskip

Our next result gives such a bound which depends only on
$\reg^2(M)$ and the Hilbert polynomial $p_U$ of the ambient
module $U$.

\begin{cor}\label{4.3}
Let $p, d, i, b$ and $r$ be as in Theorem 4.2. Let $U$ be a finitely
generated and graded $R$-module such that $\dim(U)\leq d$,
$\beg(U)\geq b$, $\reg^2(U)\leq r$ and $p_U(r)\leq p$. Then, for
each graded submodule $M\subseteq U$ such that $\reg^2(M)\leq r$ we
have
\[\max\{\reg(K^i(M)), \reg(K^i(U/M))\}\leq G^i_d(p, b, r).\]
\end{cor}

\begin{proof}
Let $M$ be as above, so that $\reg^2(M)\leq r$. Then, the short
exact sequence
\[\hskip5.2cm 0\longrightarrow M\longrightarrow U\longrightarrow
U/M\longrightarrow0 \hskip5.2cm (1)\]
implies that $\reg^2(U/M)\leq r$. Now, as previously we get on use
of Reminder 2.3(C)
 \[ \hskip3.6cm  d^0_M(r)= p_M(r), d^0_U(r)= p_U(r), d^0_{U/M}(r)= p_{U/M}(r).\hskip3.5cm
 (2)\]
As $D^1_{R_+}(M)_r\cong H^2_{R_+}(M)_r= 0$ the sequence (1) implies
\[d^0_M(r)+d^0_{U/M}(r)= d^0_U(r).\]
In view of the equalities (2) we thus get
\[p_{M}(r), p_{U/M}(r)\leq p.\]
As $\dim(M)$, $\dim(U/M)\leq d$ and $\beg(M)$, $\beg(U/M)\geq b$ we
now get the requested inequalities by Theorem 4.2.
\end{proof}

Corollary 4.3 immediately implies a bounding result which is of
the type given by Hoa-Hyry [17].

\begin{cor}\label{4.4}
Let $d, m, r\in \N$, $i\in \{0,\cdots,d\}$ and assume that $\dim
(R)\leq d$, $\reg^1(R)\leq r$ and $\dim_{R_0/\mm_0}
(R_1/\mm_0R_1)\leq m$. Let
\[\gamma:= G^i_d(\binom{m+r-1}{r-1}\length(R_0), 0, r).\]
Then, for each graded ideal $\fa\subseteq R$ with $\reg^2(\fa)\leq
r$ we have
\[\max\{\reg(K^i(\fa)), \reg(K^i(R/ \fa))\}\leq \gamma.\]
\end{cor}

\begin{proof}
Let $x_1,\cdots,x_m$ be indeterminates. Then, there is a surjective
homomorphism of graded $R_0$-algebras
$R_0[x_1,\cdots,x_m]\twoheadrightarrow R$, so that $p_R(r)\leq
\binom{m+r-1}{r-1}\length (R_0)$.

\smallskip

As $\beg(R)=0$ we now conclude by corollary 4.3.
\end{proof}

\begin{rem}\label{4.5}
\rm{If $d\geq 2$ and $R= K[x_1,\cdots,x_d]$ is a standard graded
polynomial ring over a field $K$ and $\fa\subseteq R$ is a graded
ideal with $\reg^2(\fa)\leq r$, the previous result shows that
\[\reg(K^i(R/\fa))\leq G^i_d(\binom{d+r-1}{r-1}, 0, r).\]
This inequality bounds $\reg(K^i(R/\fa))$ in terms of
$\reg^2(\fa)$. So, our result in a certain way improves [17,
Theorem 14], which bounds $\reg(K^i(R/\fa))$ only in terms of
$\reg(\fa)= \reg^1(\fa)$. On the other hand we do not insist that
our bound is sharper from the numerical point of view.\hfill
$\bullet $}
\end{rem}
Recently, "almost sharp" bounds on the Castelnuovo-Mumford
regularity in terms of the generating degree have been given by
Caviglia- Sbarra [11], Chardin-Fall-Nagel [12] and [4]. Combining
these with the previous results of the present section, we get
another type of bounding results for the Castelnuovo-Mumford
regularity of deficiency modules. Here, we restrict ourselves to
give two such bounds which hold over polynomial rings, as the
corresponding statements get comparatively simple in this case.
\begin{cor}\label{4.6}
Let $d, m\in \N$, let $i\in \{0, \cdots, d\}$, let $b,r\in \Z$,
let $R=R_0[x_1,\cdots,x_d]$ be a standard graded polynomial ring
and let $U \neq0$ be a graded $R$-module which is generated by $m$
homogeneous elements and satisfies $\beg(U)=b$ and $\reg(U)<r$.

\smallskip

Set
\begin{align*}
\varrho&:= [r+ (m+1)\length(R_0)- b]^{2^d-1},\\
\pi&:= m\binom{d+\varrho-1}{\varrho-1}\length(R_0)\ \ \text{and}\\
\delta&:= G^i_d(\pi, b, \varrho+b).
\end{align*}

Then, for each graded submodule $M\subseteq U$ with $\gendeg(M)\leq
r$ we have
\[\max\{\reg(K^i(M)), \reg(K^i(U/M))\}< \delta.\]
\end{cor}

\begin{proof}
Let $U= \sum_{i=1}^{m}Ru_i$ with $u_i\in U_{n_i}$ and $b= n_1\leq
n_2 \leq \cdots\leq n_m=\gendeg(U)\leq \reg(U)<r$.

\smallskip

As $r- b>0$ we have $r<\varrho+ b$, whence $\reg(U)<\varrho+ b$.
Therefore by Reminder 2.3(C) we obtain $p_U(\varrho+ b)=
\length(U_{\varrho+ b})$. As there is an epimorphism of graded
$R$-modules
\[\bigoplus_{i=1}^m R(-n_i)\twoheadrightarrow U\]
we thus obtain
\begin{align*}
p_U(\varrho+ b)&\leq \sum_{i=1}^{m}\binom{d+\varrho+
b-n_i-1}{\varrho+ b-n_i-1}\length(R_0)\\
&\leq m\binom{d+\varrho-1}{\varrho-1}\length(R_0)= \pi.
\end{align*}
Finally, by [4, Proposition 6.1] we have $\reg(M)\leq \varrho+b$
for each graded submodule $M\subseteq U$ with gendeg$(M)\leq r$.
Now we conclude by corollary 4.3.
\end{proof}

\begin{rem}\label{4.7}
\rm{Let $d,i>1$ and $R$ be as in Corollary 4.6 and let
$\fa\subsetneq R$ be a graded ideal of positive height. Let
\[r:= [\gendeg(\fa)(1+ \length(R_0))]^{2^d-2},\]
\[\gamma:= G^i_d(\binom{d+r-1}{r-1}\length(R_0), 0, r).\]
Then, combining [4, Corollary (5.7)(b)] with Corollary 4.4 we get
\[\max\{\reg(K^i(\fa)), \reg(K^i(R/\fa))\}< \gamma.\]
For more involved but sharper bounds of the same type one should
combine the bounds given in [12] with Corollary 4.3.\hfill
$\bullet $}
\end{rem}

Our next bound is in the spirit of the classical "problem of
finitely many steps" (cf. [15], [14]): it bounds $\reg(K^i(M))$ in
terms of the discrete data of a minimal free presentation of $M$.
Again we content ourselves to give a bounding result which is
comparatively simple and concerns only the case where $R$ is a
polynomial ring.

\begin{cor}\label{4.8}
Let $d,m\in \N$, let $i\in \{0, \cdots, d\}$, let $R=R_0[x_1, \cdots
,x_d]$ be a standard graded polynomial ring, let $p:
F\twoheadrightarrow N$ be an epimorphism of finitely generated
graded $R$-modules such that $F$ is free of rank $m>0$.

\smallskip

Set $b:=\beg(F)$ and $r:=\max\{\gendeg(F)+1, \gendeg(\ker(p))\}$
and define $\delta$ as in Corollary 4.6. Then
\[\reg(K^i(N))< \delta.\]
\end{cor}

\begin{proof}
Apply Corollary 4.6 with $U=F$ and with ker$(p)$ instead of $M$.
\end{proof}

Our last application is a bound in the spirit of Mumford's
classical result [21] which uses the Hilbert coefficients as key
bounding invariants. To formulate our result we first introduce a
few notations.

\begin{remi}\label{4.9}
(Hilbert coefficients) \rm{(A) Let $d\in \N$ and let
$\underline{e}:= (e_0, \cdots, e_{d-1})\in {\Z}^d\backslash \{0\}$.
We introduce the polynomial

\smallskip

\begin{enumerate}
\item[(i)] \hskip1cm$p_{\underline{e}}(x):= \sum_{i=0}^{d-1}(-1)^ie_i\binom{x+d-i-1}{d-i-1}\in {\fQ}[x]$
\end{enumerate}

\smallskip

which satisfies
\begin{enumerate}

\smallskip

\item[(ii)] \hskip1cm$\deg(p_{\underline{e}})= d-1-\min\{i| e_i\neq 0\}.$
\end{enumerate}

\smallskip

\noindent (B) If $M$ is a finitely generated graded $R$-module of
dimension $d$, we define the \it{Hilbert coefficients} \rm{$e_i(M)$
of $M$ for $i=0, \cdots, d-1$ such that

\smallskip

\begin{enumerate}
\item[(i)] \hskip1cm$p_M(x)=p_{(e_0(M), \cdots, e_{d-1}(M))}(x).$
\end{enumerate}

\smallskip

In particular $e_0(M)\in \N$ is the \it{Hilbert-Serre
multiplicity of} $M$. \rm{In addition we set:

\smallskip

\begin{enumerate}
\item[(ii)] \hskip1cm$e_i(M):= 0 \ \ \textrm{for all}\ \ i\in {\Z}\backslash\{0, \cdots, d-1\}.$\hfill $\bullet $
\end{enumerate}}}}
\end{remi}

\begin{nota}\label{4.10}
\rm{Let $m, d\in \N $ with $d>1$. We define a numerical function
$H^m_d: {\Z}^d\rightarrow \Z$, recursively on $d$, as follows (cf.
Reminder 4.9(A)(i))

\smallskip

\begin{enumerate}
\item[(i)] \hskip1cm$H^m_2(e_0, e_1):= 1- p_{(e_0, e_1)}(-1).$
\end{enumerate}

\smallskip

If $d>2$ and the function $H^m_{d-1}$ has already been defined,
let $\underline{e}:= (e_0, \cdots, e_{d-1})\in {\Z}^d $, set

\smallskip

\begin{enumerate}
\item[(ii)] \hskip1cm$\underline{e}':= (e_0, \cdots, e_{d-2}), \ \ \ f:=
H^m_{d-1}(\underline{e}'),$
\end{enumerate}

\smallskip

and define (cf. Reminder 4.9(A)(i))

\smallskip

\begin{enumerate}
\item[(iii)] \hskip1cm$H^m_d(\underline{e}):= \length(R_0)m\binom {f+d-3}{d-1}-p_{\underline{e}}(f-2)+f,$
\end{enumerate}

\smallskip

with the convention that $\binom{t}{d-1}=:  0$ for all $t<
d-1$.\hfill $\bullet $}
\end{nota}

\begin{rem}\label{4.11}
\rm{Let $m, d\in \N $ be with $d>1$ and set $\underline{0}:= (0,
\cdots, 0)$. Then in the notation of [9, 17.2.4], we have}
\[H^m_d= F^{(d)}_{\underline{0}}.\]\hfill $\bullet $
\end{rem}

The next result is of preliminary nature and extends [9, 17.2.7]
which at its turn generalizes Mumford bounding result (s. [21,
pg.101]).

\begin{prop}\label{4.12}
Let $d, m\in \N $ with $d>1$, let $r\in \Z$ and let $U$ be a
finitely generated graded $R$-module with $\dim(U)= d$, $\reg(U)\leq
r$ and $\dim_{R_0/\mm_0}(U_r/\mm_0U_r)\leq m$. Let $M\subseteq U$ be
a graded submodule. Then, setting $L:=U/M$, $h:= d- \dim(L)$ and
\[t:= H^m_d(m\length (R_0)- (-1)^he_{-h}(L(r)), (-1)^he_{1-h}(L(r)),
 \cdots, (-1)^he_{d-1-h}(L(r)),\] we have

\smallskip

\begin{enumerate}
\item[\rm{(a)}] \hskip1cm$\reg^1(L)\leq \max\{0, t-1\}+r$;
\end{enumerate}

\smallskip

\begin{enumerate}
\item[\rm{(b)}]\hskip1cm $\reg^2(M)\leq \max\{1, t\}+r$.
\end{enumerate}

\smallskip

\end{prop}

\begin{proof}
If $M$ is $R_+$-torsion, we have $\reg^1(L)= \reg^1(U)\leq r$ and
$\reg^2(M)= -\infty$ so that our claim is obvious. Therefore we may
assume that $M$ is not $R_+$-torsion.

\smallskip

We may assume that $R_0/\m_0$ is infinite.
We may in addition replace $R$ by $R/(0:_RU)$ and hence assume
that $\dim(R)= d$. We now find elements $a_1, \cdots, a_d\in R_1$
which form a system of parameters for $R$. In particular $R$ is a
finite integral extension of $R_0[a_1, \cdots, a_d]$. Consider the
polynomial ring $R_0[x_1, \cdots, x_d]$ and the homomorphism of
$R_0$-algebras $f: R_0[x_1, \cdots, x_d]\rightarrow R$ given by
$x_i\mapsto a_i$ for $i=1, \cdots, d$. Then, $M$ is a finitely
generated graded module over $R_0[x_1, \cdots, x_d]$ and
$\sqrt{R_+}= \sqrt{(x_1, \cdots, x_d)R}$.

\smallskip

So, the numerical invariants of $U$ and $M$ which occur in our
statement do not change if we consider $U$ and $M$ as $R_0[x_1,
\cdots, x_d]$-modules by means of $f$. Therefore, we may assume
that $R= R_0[x_1, \cdots, x_d]$. Now, we have $\gendeg(U(r))\leq
\reg(U(r))\leq 0$ and $\dim_{R_0/\m_0}(U(r)_0/\m_0U(r)_0)\leq m$.
This implies that the $R$-module  $U(r)_{\geq0}$ is generated by
(at most) $m$ homogeneous elements of degree $0$. Therefore we
have an epimorphism of graded $R$-modules
\[R^{\bigoplus m}\stackrel{g}{\twoheadrightarrow} U(r)_{\geq0}.\]
Let $N:= g^{-1}(M(r)_{\geq0})$. As $M$ is not $R_+$-torsion we
have $M(r)_{\geq0}\neq 0$ and hence $N\neq 0$. As $N\subseteq
R^{\bigoplus m}$ and by our choice of $R$ we thus have $\dim(N)=
d$. Now, the isomorphism of graded $R$-modules $R^{\bigoplus m}/
N \cong (L(r))_{\geq0}$ implies
\\

\smallskip
$m\length(R_0)\binom{x+d-1}{d-1}-\sum_{i=0}^{d-1}(-1)^ie_i(N)\binom{x+d-i-1}{d-i-1}$
\begin{align*}
&=p_{R^{\bigoplus m}}(x)- p_N(x)= p_{R^{\bigoplus
m}/N}(x)=p_{(L(r))_{\geq0}}(x)=p_{L(r)}(x)
\\&=\sum_{j=0}^{d-h-1}(-1)^je_j(L(r))\binom{x+d-h-j-1}{d-h-j-1}
=\sum_{i=0}^{d-1}(-1)^{i-h}e_{i-h}(L(r))\binom{x+d-i-1}{d-i-1}.
\end{align*}
Therefore
\[e_0(N)= m\length(R_0)- (-1)^he_{-h}(L(r))\]
and
\[e_i(N)= (-1)^he_{i-h}(L(r))\ \ \textrm{ for all}\ \  i\in\{1, \cdots,
d-1\}.\] So, according to [9, 17.2.7] and Remark 4.11 we obtain
\begin{align*}
\reg^2(N)&\leq F_0^{(d)}(e_0(N), \cdots, e_{d-1}(N))\\
 &= H^m_d(m \length (R_0)-
(-1)^he_{-h}(L(r)), (-1)^he_{1-h}(L(r)), \cdots,
(-1)^he_{d-1-h}(L(r)))=: t. \end{align*}

Now, the short exact sequence of
graded $R$-modules
\[0\longrightarrow N\longrightarrow R^{\bigoplus m}\longrightarrow (U(r)/M(r))_{\geq0}\longrightarrow 0\]
implies $\reg^1((U(r)/M(r))_{\geq0})\leq \max\{0, t-1\}$, whence
$\reg^1(U(r)/M(r))\leq \max\{0, t-1\}$, so that finally
\[\reg^1(L)= \reg^1(U(r)/M(r))+ r\leq \max\{0, t-1\}+r\]
and
\begin{align*}
\reg^2(M)= \reg^2(M(r))+ r &\leq \max\{\reg^2(U(r)), \reg^1(U(r)/M(r))+ 1\}+ r\\
&\leq \max\{0, \max\{0, t-1\}+ 1\}+ r\leq \max\{1, t\}+ r.
\end{align*}
This proves our claim.

\end{proof}

Now, we may bound the Castelnuovo-Mumford regularity of deficiency
modules as follows:

\begin{cor}\label{4.13}
Let the notations and hypothesis be as in Proposition 4.12. In
addition let $b\in \Z$ and $p\in {\N}_0$ such that $\beg(U)\geq b$
and $p_U(r)\leq p$.

Then, for all $i\in \{0, \cdots, d\}$ we have
\[\max\{\reg(K^i(M)), \reg(K^i(U/M))\}\leq G^i_d(p, b, \max\{1, t\}+r).\]
\end{cor}

\begin{proof}
This is clear by Corollary 4.3 and Proposition 4.12(b).
\end{proof}

\rm{Applying this to the "classical" situation of [21] where $M=
\fa$ is a graded ideal of a polynomial ring we finally can say}

\begin{cor}\label{4.14}
Let $R= R_0[x_1, \cdots, x_d]$ be a standard graded polynomial ring
with $d>1$ and let $\fa\subseteq R$ be a graded ideal. Set $h:=
\height(\fa)$ and
\[t:= H^1_d(\length(R_0)-(-1)^he_{-h}(R/\fa), (-1)^he_{1-h}(R/\fa), \cdots, (-1)^he_{d-1-h}(R/\fa)).\]
Then, for all $i\in \{0, \cdots, d\}$ we have
\[\max\{\reg(K^i(\fa)), \reg(K^i(R/\fa))\}\leq G^i_d(1, 0, \max\{1, t\}).\]
\end{cor}

\begin{proof}
Choose $U:= R, M:= \fa, m=1, r=0, b=0, p=1$. Observe also that $d-
\dim(R/\fa)= h$ and apply Corollary 4.13.
\end{proof}

\vskip 1 cm
\section{ Bounding Cohomological Postulation Numbers}

\rm{In [6, Theorem 4.6] it is shown that the cohomological
postulation numbers of a projective scheme $X$ over a field $K$
with respect to a coherent sheaf of $\lO_X$-modules $\cF$ are
bounded by the cohomology diagonal $(h^i(X,
\cF(-i)))_{i=0}^{\dim(\cF)}$ of $\cF$. On use of Theorem 3.6 this
"purely diagonal bound" now can be generalized to the case where
the base field $K$ is replaced by an arbitrary Artinian ring. To
do so, we first introduce some appropriate notions.}

\begin{defin}\label{5.1}
\rm{For $d\in {\N}$ and $i\in \{0,\cdots,d-1\}$ we define the
bounding function
\[E^i_d: {\N}_0^d\rightarrow {\Z}\]
by
\[E^i_d(x_0,\cdots,x_{d-1}):= -F^{i+1}_{d}(x_0,\cdots,x_{d-1}, 0),\]
where $F^{i+1}_{d}$ is defined according to Definition 3.4.\hfill
$\bullet $}
\end{defin}

\begin{defin}\label{5.2}
\rm{Let $d\in {\N}$. By $\D^d$ we denote the class of all pairs $(R,
M)$ in which $R = \oplus_{n\in {\N}_0}R_n$ is a Noetherian
homogenous ring with Artinian base ring $R_0$ and $M = \oplus_{n\in
{\Z}}M_n$ is a finitely generated graded $R$-module with
$\dim(M)\leq d$.\hfill $\bullet $}
\end{defin}
\rm{Now, we are ready to state the announced "purely diagonal"
bounding result as follows:}

\begin{thm}\label{5.3}
Let $d\in {\N}$, let $x_0,\cdots,x_{d-1}\in {\N}_0$ and let $(R,
M)\in \D^d$ such that $d^j_M(-j)\leq x_j$ for all $j\in
\{0,\cdots,d-1\}$. Then for all $i\in \{0,\cdots,d-1\}$ we have
\[\nu^i_M\geq E^i_d(x_0,\cdots,x_{d-1}).\]
\end{thm}
\begin{proof}
On use of standard reduction arguments and the monotonicity
statement of Remark 3.5(A) we can restrict ourselves to the case
where the Artinian base ring $R_0$ is local. Consider the graded
submodule $N:= M_{\geq 0}= \oplus_{n\geq 0}M_n$ of $M$. As the
module $M/N$ is $R_+$-torsion, the graded short exact sequence
$0\longrightarrow N\longrightarrow M\longrightarrow
M/N\longrightarrow0$ yields isomorphisms of graded $R$-modules
$D^j_{R_+}(M)\cong D^j_{R_+}(N)$ and hence equalities $d^j_M=
d^j_N$ for all $j\in {\N}_0$. These allow to replace $M$ by $N$
and hence to assume that $\beg(M)\geq 0$.

\smallskip

Now, on use of Corollary 3.7 and Reminders 2.6(C)(vi) and 2.3(C)
we get
\[\nu^i_M= -p(K^{i+1}(M))\geq -\reg(K^{i+1}(M))
\geq -F^{i+1}_d(x_0,\cdots,x_{d-1}, 0)= E^i_d(x_0,\cdots,x_{d-1}).\]
\end{proof}

\rm{As a consequence of Theorem 5.3 we get the following
finiteness result which is shown in [6] for the special case of
homogeneous rings $R$ whose base rings $R_0$ are field.

\begin{thm}\label{5.4}
Let $d\in \N$ and let $x_0, \cdots, x_{d-1}\in {\N}_0$. Then, the
set of cohomological Hilbert functions
\[\{d^i_M \mid i\in {\N}_0; (R,M)\in \D^d; d^j_M(-j)\leq x_j \ \ \text{for}\ \ j=0, \cdots, d-1\}\]
is finite.
\end{thm}

\begin{proof}
First, we set
\[\D:= \{(R,M)\in \D^d \mid d^j_M(-j)\leq x_j \ \ \text{for}\ \ j=0, \cdots, d-1\}.\]
As $d^i_M\equiv 0$ if $(R,M)\in \D^d$ and $i\geq d$, it suffices to
show that the set
\[\{d^i_M \mid i<d, (R,M)\in \D\}\]
is finite.

\smallskip

According to [8, Lemma 4.2] we have
\[\hskip3.9cm d^i_M(n)\leq
\sum_{j=0}^{i}\binom{-n-j-1}{i-j}\biggr[\sum_{l=0}^{i-j}\binom{i-j}{l}x_{i-l}\biggr]\hskip3.8cm
(1)\]
for all $i\in {\N}_0$, all $n\leq -i$ and all $(R,M)\in \D$.
According to Theorem 5.3 there is some integer $c\leq -d+1$ such
that $\nu^i_M>c$ for all $(R,M)\in \D$ and all $i<d$. So, using
the notation of Reminder 2.3(B) we have $q^i_M(n)= d^i_M(n)$ for
all $i<d$ and all $n\leq c$.

\smallskip

As $\deg(q^i_M)\leq i$ (s. Reminder 2.3 (B)(ii),(iv)) it follows
from (1) that the set
\[\{q^i_M \mid i<d, (R,M)\in \D\}\]
is finite. Consequently, the set
\[\{d^i_M(n) \mid i<d, n\leq c, (R,M)\in \D\}\]
is finite, too. So, in view of (1) the set
\[\{d^i_M(n) \mid i<d, n\leq -i, (R,M)\in \D\}\]
must be finite. It thus remains to show that for each $i<d$ the set
\[S_i:= \{d^i_M(n) \mid n\geq -i, (R,M)\in \D\}\]
is finite. To this end, we fix $i\in \{1, \cdots, d-1\}$.
According to [7, Corollary (3.11)] there are two integers
$\alpha, \beta$ such that
\[\hskip3.8cm d^i_M(n)\leq \alpha \ \ \textrm{for all} \ \ n\geq -i\ \ \textrm{and
all} \ \ (R,M)\in \D,\hskip3.7cm(2)\]
\[\hskip5.1cm \reg^2(M)\leq \beta\ \ \textrm{for all} \ \ (R,M)\in
\D.\hskip5cm (3)\]
The inequality (3) implies that $d^i_M(n)=0$ for all $n\geq
\beta-i+1$ and hence by (2) the set $S_i$ is finite.

\smallskip

It remains to show that the set $S_0$ is finite.

\smallskip

To do so, we write $\overline{M}:= D_{R_+}(M)_{\geq 0}$ for all
pairs $(R, M)\in \D$. As $(D_{R_+}(M)/ \overline{M})_{\geq 0}=0$,
$H^k_{R_+}(D_{R_+}(M))=0$ for $k=0, 1$ and
$D^j_{R_+}(D_{R_+}(M))\cong D^j_{R_+}(M)$ for all $j\in {\N}_0$ we
get $ \Gamma_{R_+}(\overline{M})= 0$, $\en(H^1_{R_+}(M))< 0 $ and
$d^j_M\equiv d^j_{\overline{M}}$ for all $j\in {\N}_0$ and for
all $(R, M)\in \D$. In particular $(R, \overline{M})\in \D$ for
all $(R, M)\in \D$. So, writing
\[\overline{\D}:= \{(R,M)\in \D | \Gamma_{R_+}(M)= 0, \en(H^1_{R_+}(M))< 0\}\]
it suffices to show that the set
\[\overline{S_0}=\{d^0_M(n) | n\geq0, (R,M)\in \overline{\D}\}\]
is finite.

\smallskip

If $(R,M)\in \overline{\D}$ we conclude by statement (3) that
$p(M)\leq \reg(M)= \reg^1(M)= \max\{\en(H^1_{R_+}(M))+1,
\reg^2(M)\}\leq \max\{0, \beta\}:= \beta'$. As $\deg(p_M)<d$ it
follows by statement (2) that the set of Hilbert polynomials
$\{p_M | (R,M)\in \overline{\D}\}$ is finite. Consequently, the
set $\{d^0_M(n) | n>\beta', (R,M)\in \overline{\D}\}$ is finite.
Another use of statement (2) now implies the finiteness of
$\overline{S_0}$.
\end{proof}

\begin{cor}\label{5.5}
Let the notations be as in Theorem 5.4 and Reminder 2.3. Then the
sets of polynomials
\[\{q^i_M \mid i\in {\N}_0; (R,M)\in \D^d; d^j_M(-j)\leq x_j \ \ \text{for}\ \ j=0, \cdots, d-1\},\]
\[\{p_M \mid (R,M)\in \D^d; d^j_M(-j)\leq x_j \ \ \text{for}\ \ j=0, \cdots, d-1\}\]
are finite.
\end{cor}

\begin{proof}
This is clear by Theorem 5.4.
\end{proof}

\vskip 1 cm

{\bf Acknowledgment.}{The third author would like to thank the
Institute of Mathematics of University of Z\"urich for financial
support and hospitality during the preparation of this paper.}

\renewcommand{\descriptionlabel}[1]%
             {\hspace{\labelsep}\textrm{#1}}
\begin{description}
\setlength{\labelwidth}{12mm} \setlength{\labelsep}{1.3mm}
\setlength{\itemindent}{0mm}

\item [{[1]}] D. Bayer and D.Mumford, \emph{What can be computed in
algebraic geometry?} in "Computational Algebraic Geometry and
Commutative Algera" Proc. Cortona 1991 (D.Eisenbud and L.
Robbiano, Eds.) Cambridge University Press (1993) 1-48.

\item [{[2]}]  M. Brodmann,
\emph{Cohomological invariants of coherent sheaves over
projective schemes - a survey}, in ``Local Cohomology and its
Applications'' (G. Lyubeznik, Ed), 91-120, M. Dekker Lecture
Notes in Pure and Applied Mathematics {\bf226} (2001).

\item [{[3]}] M. Brodmann, \emph{Castelnouvo-Mumford regularity and degrees of generators of graded submodules},
 Illinois J. Math. {\bf47}, no. 3 (Fall 2003).

\item [{[4]}] M.Brodmann and T.G\"otsch, \emph{Bounds for the
Castelnuovo-Mumford regularity}, to appear in Journal of
Commutative Algebra.

\item [{[5]}]
M. Brodmann, M. Hellus, {\em Cohomological patterns of coherent
sheaves over projective schemes}, Journal of Pure and Applied
Algebra {\bf 172} (2002), 165-182.

\item [{[6]}]
M. Brodmann, F. A. Lashgari, {\em A diagonal bound for
cohomological postulation numbers of projective schemes}, J.
Algebra {\bf 265} (2003), 631-650.

\item [{[7]}]
M. Brodmann, C. Matteotti and N. D. Minh, {\em Bounds for
cohomological Hibert functions of   projective schemes over
artinian rings}, Vietnam Journal of Mathematics {\bf
28}(4)(2000), 345-384.

\item [{[8]}]
M. Brodmann, C. Matteotti and N. D. Minh, {\em Bounds for
cohomological deficiency functions  of   projective schemes over
acrtinian rings}, Vietnam Journal of Mathematics {\bf
31}(1)(2003), 71-113.

\item [{[9]}]
M. Brodmann and R.Y. Sharp, {\em Local cohomology: an algebraic
introduction with geometric applications}, Cambridge University
Press (1998).

\item [{[10]}]
G. Caviglia, {\em Bounds on the Castelnuovo-Mumford regularity of
tensor products}, Proc. AMS. {\bf135} (2007) 1949-1957.

\item [{[11]}]
G.Caviglia and E.Sbarra, \emph{Characteristic free bounds for the
Castelnuovo-Mumford regularity}, Compos. Math. {\bf141} (2005)
1365-1373.

\item [{[12]}]
M.Chardin, A.L.Fall and U.Nagel, \emph{Bounds for the
Castelnuovo-Mumford regularity of modules}, Math. Z. {\bf258}
(2008) 69-80.

\item [{[13]}]
A. Grothendieck, S\'eminare de g\'eometrie alg\'ebrque IV,
Speringer Lecture Notes 225, Speringer (1971).

\item [{[14]}]
K.Hentzelt and E.Noether, \emph{Zur Theorie der Polynomideale und
Resultanten}, Math. Ann. {\bf88} (1923) 53-79.

\item [{[15]}]
G.Hermann, \emph{\"Uber die Frage der endlich vielen Schritte in
der Theorie der Polynomideale}, Math. Ann. {\bf95} (1926) 736-788.

\item [{[16]}] L.T.Hoa,
\emph{Finiteness of Hilbert functions and bounds for the
Castelnuovo-Mumford regularity of initial ideals}, Trans. AMS.
{\bf360} (2008) 4519-4540.

\item [{[17]}]
L. T. Hoa and E. Hyry {\em Castelnuovo-Mumford regularity of
canonical and deficiency modules}, J. Algebra {\bf 305} (2006)
no.2, 877-900.

\item [{[18]}]
J. Kleiman  , \emph{Towards a numerical theory of ampleness},
Annals of Math. {\bf84} (1966) 293-344.

\item [{[19]}]
C. H. Linh, {\em   Upper bound for Castelnuovo-Mumford regularity
of associated graded modules},  Comm. Algebra, {\bf 33}(6) (2005),
1817-1831.

\item[{[20]}] E.W.Mayr and A.R.Meyer, \emph{The complexity of the
word problem for commutative semigroups and polynomial ideals},
Advances in Math. {\bf46} (1982) 305-329.

\item [{[21]}]
D.Mumford, \emph{Lectures on curves on an algebraic surface},
Annals of Math.Studies {\bf59}, Princeton University Press (1966).

\item [{[22]}]
M.E. Rossi, N.V. Trung and G. Valla, {\em Castelnuovo-Mumford
regularity and extended degree}, Trans. Amer. Math. Soc. {\bf 355}
(2003), no. 5, 1773-1786.

\item [{[23]}]
P. Schenzel, {\em Dualisierende Komplexe in der lokalen Algebra
und Buchsbaumringe}. Lecture Notes in Mathematics, 907.
Springer-Verlag, Berlin-New York, 1982. MR 83i:13013.

\item [{[24]}]
P.Schenzel, \emph{On birational Macaulayfcations and
Cohen-Macaulay canonical modules}, J. Algebra {\bf275} (2004)
751-770.
\end{description}
\end{document}